\input ./style/arxiv-general.cfg
\documentclass[MSNbibl,citesort,number,seceqn,dvips]{arxbj}
\makeatletter
   \@ifpackageloaded{graphicx}{}{\usepackage{graphicx}}
\makeatother

\aid{0}
\volume{21}
\issue{3}
\pubyear{2015}
\firstpage{1600}
\lastpage{1628}
\doi{10.3150/14-BEJ615} 

\makeatletter

\newcommand{\rright}{\right}
\newcommand{\lleft}{\left}
\newcommand{\rrVert}{\Vert}
\newcommand{\llVert}{\Vert}

\newtheorem{them}{Theorem}[section]
\newtheorem{lem}[them]{Lemma}
\newtheorem{lemma}{Lemma}
\newtheorem{pro}[them]{Proposition}
\newremark{rem}[them]{Remark}
\newremark{remm}[lemma]{Remark}
\newproclaim{Assumption}{Assumption}

\newcommand{\tr}{\operatorname{tr}}
\newcommand{\Var}{\operatorname{Var}}
\renewcommand{\epsilon}{\varepsilon}
\renewcommand{\emptyset}{\varnothing}

\def\sfrac#1#2{#1/#2}
\def\vfrac#1#2{(#1)/#2}

\def\sklfrac#1#2{(#1/#2)}

\def\sklafrac#1#2{(#1/(#2))}

\makeatother

\begin{document}
\begin{frontmatter}

\title{The logarithmic law of random determinant}
\runtitle{The logarithmic law of random determinant}

\begin{aug}
\author[1]{\inits{Z.}\fnms{Zhigang} \snm{Bao}\corref{}\thanksref{1}\ead[label=e1]{maomie2007@gmail.com}},
\author[2]{\inits{G.}\fnms{Guangming} \snm{Pan}\thanksref{2}\ead[label=e2]{gmpan@ntu.edu.sg}} \and
\author[3]{\inits{W.}\fnms{Wang} \snm{Zhou}\thanksref{3}\ead[label=e3]{stazw@nus.edu.sg}}
\address[1]{Department of Mathematics, Zhejiang University, Hangzhou,
P.R. China.\\
\printead{e1}}
\address[2]{School of Physical and Mathematical Sciences, Nanyang
Technological University, Singapore.\\
\printead{e2}}
\address[3]{Department of Statistics and Applied Probability, National
University of
Singapore, Singapore.\\
\printead{e3}}
\end{aug}

\received{\smonth{5} \syear{2013}}
\revised{\smonth{1} \syear{2014}}

%
\begin{abstract}
Consider the square random matrix $A_n=(a_{ij})_{n, n}$, where $\{
a_{ij}:=a_{ij}^{(n)},i,j=1,\ldots,n\}$ is a collection of independent
real random variables with means zero and variances one. Under the
additional moment condition
\[
\sup_n\max_{1\leq i,j\leq n}\mathbb{E}a_{ij}^4<
\infty,
\]
we prove Girko's logarithmic law of $\det A_n$ in the sense that as
$n\rightarrow\infty$
\begin{eqnarray*}
\frac{\log|\det A_n|-\sklfrac12\log(n-1)!}{\sqrt{\sklfrac12\log
n}}\stackrel{d}\longrightarrow N(0,1).
\end{eqnarray*}
\end{abstract}

%

\end{frontmatter}

\section{Introduction}\label{sec1}
Consider the square random matrix $A_n=(a_{ij})_{n,n}$, where $\{
a_{ij}:=a_{ij}^{(n)},i,j=1,\ldots,n\}$ is a collection of independent
real random variables with means zero and variances one. Moreover, we assume
\begin{eqnarray*}
\sup_{n}\max_{1\leq i,j\leq n}\mathbb{E}a_{ij}^4<
\infty.
\end{eqnarray*}
The main purpose of this paper is to study the determinant of $A_n$. As
an important and fundamental function of a matrix, the random
determinant has been investigated in many articles. For instance, the
study of the moments of random determinants arose in the 1950s. One can
refer to Dembo \cite{Dee}, Forsythe and Tukey \cite{FT},
Nyquist, Rice and Riordan \cite{NRR}, Pr{\'e}kopa \cite{Pr} for this
topic. Besides, some lower and upper bounds for the magnitudes of
random determinants were obtained in Costello and Vu \cite{CV} and
Tao and Vu \cite{TV1}
recently. A basic problem in the random determinant theory is to derive
the fluctuation of the quantity $\log|\det A_n|$, which can give us an
explicit description of the limiting behaviour of $|\det A_n|$.
Particularly, when the entries of $A_n$ are i.i.d. Gaussian, Goodman \cite{Go} found that $\det A_n^2$ can be written as a
product of $n$
independent $\chi^2$ variables with different degrees of freedom. In
fact, by using the Householder transform repeatedly, one can get that
the joint distribution of the eigenvalues of $A_nA_n^T$ is the same as
that of the tridiagonal matrix
$
L_n=D_nD_n^T$,
where
\begin{eqnarray*}
D_n=\lleft( %
\begin{array} {c@{\quad}c@{\quad}c@{\quad}c}
a_n & & &
\\
b_{n-1} & a_{n-1} & &
\\
&\cdots&\cdots&
\\
& & b_1 & a_1
\\
\end{array} %
\rright).
\end{eqnarray*}
Here $\{a_n,\ldots,a_1,b_{n-1},\ldots,b_1\}$ is a collection of
independent variables such that $a_i\sim\chi_i, b_j\sim\chi_j$ for
$i=1,\ldots,n, j=1,\ldots, n-1$. Such a tridiagonal form is well
known for Gaussian matrix. One can refer to Dumitriu and Edelman \cite
{DE}, for instance.
Then apparently one has
\begin{eqnarray*}
\det A_n^2=\det A_nA_n^T
\stackrel{d}=\prod_{i=1}^n
a_i^2,
\end{eqnarray*}
which implies that $\log|\det A_n|$ can be represented by a sum of $n$
independent random variables. Then by the elementary properties of
$\chi^2$ distributions, the following CLT can be obtained
%
\begin{eqnarray}\label{1.1}
\frac{\log|\det A_n|-\sklfrac{1}2\log(n-1)!}{\sqrt{\sklfrac{1}2\log
n}}\stackrel{d}\longrightarrow N(0,1).
\end{eqnarray}
For details, one can see Rouault \cite{Ro} or the Appendix of
Costello and Vu \cite{CV} for
instance. Like most of topics in the Random Matrix Theory, one may ask
whether there is a ``universal'' phenomenon for the CLT of $\log|\det
A_n|$ under general distribution assumption. The best result on this
problem was given by Girko in \cite{Girko} (one
can also refer to
Girko's books \cite{Girko1} and  \cite{Girko2} or his paper  \cite
{Girko79} for his former results on this topic), where the author only
required the existence of the $(4+\delta)$th moment of the entries
for some positive $\delta$. Girko named (\ref{1.1}) as ``the
logarithmic law'' for random determinant. In Girko \cite
{Girko}, using an
elegant ``method of perpendiculars'' and combining the classical CLT
for martingale, Girko claimed (\ref{1.1}) is universal under the
moment assumption mentioned above. However, though the proof route of
Girko \cite{Girko} is clear and quite original, it seems
the proof is not
complete and several parts are lack of mathematical rigour. Recently,
 Nguyen and Vu \cite{NV} provided a transparent proof
of (\ref{1.1})
for the general distribution case, under much stronger moment
assumption in the sense that for all $t>0$,
%
\begin{eqnarray}\label{1.2}
\mathbb{P}\bigl(|a_{ij}|\geq t\bigr)\leq C_1\exp
\bigl(-t^{C_2}\bigr)
\end{eqnarray}
with some positive constants $C_1,C_2$ independent of $i,j,n$.
Obviously, (\ref{1.2}) implies the existence of the moment of any
order. The basic framework of the proof in Nguyen and Vu \cite{NV}
is similar to
Girko's method of perpendiculars, which will be introduced in the next
section. However, in order to provide a transparent proof, the authors
of Nguyen and Vu \cite{NV} inserted a lot of new ingredients.
Moreover, some
unrigorous steps in Girko \cite{Girko} can be fixed by
the methods provided
in Nguyen and Vu \cite{NV}. One may find that Nguyen and Vu
\cite{NV} also provided a convergence
rate of the logarithmic law as $\log^{-1/3+\mathrm{o}(1)}n$, which is
nearly optimal.

In this paper, also relying on the basic strategy of Girko's method of
perpendiculars, we will provide a complete and rigorous proof under a
weaker moment condition. More precisely, we only require the existence
of $4$th moment of the matrix entries. Our main result is

\begin{them} \label{thm.1.1}Let $A_n=(a_{ij})_{n,n}$ be a square random
matrices, where $\{a_{ij},1\leq i,j\leq n\}$ is a collection of
independent real random variables with common mean $0$ and variance
$1$. Moreover, we assume
\begin{eqnarray*}
\sup_{n}\max_{1\leq i,j\leq n}\mathbb{E}a_{ij}^4<
\infty.
\end{eqnarray*}
Then we have the logarithmic law for $|\det A_n|$: as $n$ tends to infinity,
\begin{eqnarray*}
\frac{\log|\det A_n|-\sklfrac{1}2\log(n-1)!}{\sqrt{\sklfrac{1}2\log
n}}\stackrel{d}\longrightarrow N(0,1).
\end{eqnarray*}
\end{them}

Our paper is organized as follows. In Section~\ref{sec2}, we will sketch the
main idea of Girko's method of perpendiculars and the proof route of
Theorem~\ref{thm.1.1}. In Sections~\ref{sec3}, \ref{sec4} and \ref{sec5}, we will present the
details of the proof with the aid of some additional lemmas, whose
proofs will be given in the \hyperref[appA]{Appendix}.

Throughout the paper, the notation such as $C,K$ will be used to denote
some positive constants independent of $n$, whose values may differ
from line to line.
We use $\llVert \cdot\rrVert _2$ and $\llVert \cdot
\rrVert _{\mathrm{op}}$ to represent the Euclidean
norm of a vector and the operator norm of a matrix respectively as usual.

\section{Girko's method of perpendiculars}\label{sec2}
In this section, we will sketch the main framework of Girko's method of
perpendiculars, which was also pursued by Nguyen and Vu in the recent
work  \cite{NV}. To state this method rigorously, we
need the following
proposition whose proof will be given in Appendix \hyperref[appB]{B}.

\begin{pro} \label{pro.1.2} For the matrix $A_n$ defined in Theorem~\ref{1.1}, we can find a modified matrix $A'_n=(a'_{ij})_{n,n}$
satisfying the assumptions in Theorem~\ref{1.1} such that
%
\begin{eqnarray}\label{1.1.1.2}
\mathbb{P}\bigl\{\mbox{all square submatrices of }A'_n\mbox{ are
invertible}\bigr\}=1
\end{eqnarray}
and
%
\begin{eqnarray}\label{1.1.1.1}
\mathbb{P}\bigl\{\log|\det A_n|-\log\bigl|\det A'_n\bigr|=
\mathrm{o}(\sqrt{\log n})\bigr\} =1-\mathrm{o}(1).
\end{eqnarray}
\end{pro}

\begin{rem} The construction of the modified matrix $A'_n$ can be found
in Nguyen and Vu \cite{NV}. The strategy is to set
$a'_{ij}:=(1-\epsilon
^2)^{1/2}a_{ij}+\epsilon\theta_{ij}$, where $\{\theta_{ij},1\leq
i,j\leq n\}$ is a collection of independent bounded continuous random
variables with common mean zero and variance one and is independent of
$A_n$. By choosing $\epsilon$ to be extremely small, say $n^{-Kn}$ for
large enough constant $K>0$, it was shown in Nguyen and Vu \cite
{NV} that (\ref
{1.1.1.1}) holds. The proof in Nguyen and Vu \cite{NV} relies on a
lower bound
estimate on the smallest singular value of square random matrices
provided in Theorem~2.1 of Tao and Vu \cite{TV}.
To adapt to our condition, we will use Theorem~4.1 of G{\"o}tze and Tikhomirov
\cite{GT}
instead (by choosing $p_n=1$ in G{\"o}tze and Tikhomirov \cite{GT}). For
convenience of the
reader, we sketch the proof of Proposition~\ref{pro.1.2} in
Appendix \hyperref[appB]{B}. The proof is just a slight modification of that in Nguyen and Vu
\cite{NV} under
our setting and assumptions.
\end{rem}

Therefore, with the aid of Proposition~\ref{pro.1.2}, we can always
work under the following assumption.

{\renewcommand{\theAssumption}{C\tsub{0}}
\begin{Assumption}\label{assC} We assume that
$A_n=(a_{ij})_{n,n}$ is a square random matrix, where $\{a_{ij},1\leq
i,j\leq n\}$ is a collection of independent real random variables with
common mean zero and variance one. Besides, $\sup_n\max_{1\leq i,j\leq
n}\mathbb{E} a_{ij}^4<\infty$. Moreover, all square submatrices of
$A_n$ are invertible with probability one.
\end{Assumption}}

The starting point of the method of perpendiculars is the elementary
fact that the magnitude of the determinant of $n$ real vectors in $n$
dimensions is equal to the volume of the parallelepiped spanned by
those vectors. Therefore, by the basic ``base times height'' formula,
one can represent $|\det A_n|$ by the products of $n$ perpendiculars.
To make it more precise, we introduce some notations at first.

In the sequel, we will use $\mathbf{a}_k^T$ to denote the $k$th row of
$A_n$. And let $A_{(k)}$ be the $k\times n$ rectangular matrix formed
by the first $k$ rows of $A_n$. Particularly, one has $A_{(1)}=\mathbf
{a}_1^T$ and $A_{(n)}=A_n$. Moreover, we use the notation $V_i$ to
denote the subspace generated by the first $i$ rows of $A_n$ and
$P_i=(p_{jk}(i))_{n,n}$ to denote the projection matrix onto the space
$V_i^\bot$. Let $\gamma_{i+1}$ be the distance from $\mathbf
{a}_{i+1}^T$ to $V_i$ for $1\leq i\leq n-1$. And we set $\gamma
_1=\llVert \mathbf{a}_1^T\rrVert _2$. Then by the ``base
times height'' formula,
we can write
%
\begin{eqnarray}\label{2.0}
\det A_n^2=\prod_{i=0}^{n-1}
\gamma_{i+1}^2.
\end{eqnarray}
Observe that $\gamma_{i+1}$ is the norm of the projection of $\mathbf
{a}_{i+1}^T$ onto $V_i^\bot$. Thus we also have
%
\begin{eqnarray}\label{2.1}
\gamma_{i+1}^2=\mathbf{a}_{i+1}^TP_i
\mathbf{a}_{i+1},\qquad  1\leq i\leq n-1.
\end{eqnarray}
Moreover, by the definition of $V_i$ and Assumption \textup{\ref{assC}} one
has that with probability one $A_{(i)}A_{(i)}^T$ is invertible and
%
\begin{eqnarray}\label{2.5}
P_i=I_n-A_{(i)}^T
\bigl(A_{(i)}A_{(i)}^T\bigr)^{-1}A_{(i)}.
\end{eqnarray}
Then a direct consequence of the definition of $\gamma_{i+1}$ is
\begin{eqnarray*}
\mathbb{E}\bigl\{\gamma_{i+1}^2|P_i\bigr
\}=\tr P_i=n-i,\qquad  0\leq i\leq n-1.
\end{eqnarray*}
It follows from (\ref{2.0}) that
\begin{eqnarray*}
\log\det A_n^2=\sum_{i=0}^{n-1}
\log\gamma_{i+1}^2,
\end{eqnarray*}
which yields
%
\begin{eqnarray}\label{2.2}
\log\det A_n^2-\log(n-1)!=\sum
_{i=0}^{n-1}\log\frac{\gamma
_{i+1}^2}{n-i}+\log n.
\end{eqnarray}
Now we set
\begin{eqnarray*}
X_{i+1}:=X_{n,i+1}=\frac{\gamma_{i+1}^2-(n-i)}{n-i}.
\end{eqnarray*}
And we write
%
\begin{eqnarray}\label{2.4}
\log\frac{\gamma_{i+1}^2}{n-i}=X_{i+1}-\frac{X_{i+1}^2}{2}+R_{i+1},
\end{eqnarray}
where
\begin{eqnarray*}
R_{i+1}:=\log(1+X_{i+1})-\biggl(X_{i+1}-
\frac{X_{i+1}^2}{2}\biggr).
\end{eqnarray*}
Then by (\ref{2.2}), one can write
%
\begin{eqnarray}\label{2.3}
&&\frac{\log\det A_n^2-\log(n-1)!}{\sqrt{2\log n}}
\nonumber
\\[-8pt]\\[-8pt]
&&\quad =\frac{1}{\sqrt{2\log n}}\sum_{i=0}^{n-1}X_{i+1}-
\frac{1}{\sqrt {2\log n}}\Biggl(\sum_{i=0}^{n-1}
\frac{X_{i+1}^2}{2}-\log n\Biggr)+\frac
{1}{\sqrt{2\log n}}\sum
_{i=0}^{n-1}R_{i+1}.
\nonumber
\end{eqnarray}

Crudely speaking, the main route is to prove that the first term of
(\ref{2.3}) weakly converges to the standard Gaussian distribution and
the remaining two terms tend to zero in probability. Let $\mathcal
{E}_i$ be the $\sigma$-algebra generated by the first $i$ rows of
$A_n$, by definition we have
\begin{eqnarray*}
\mathbb{E}\{X_{i+1}|\mathcal{E}_i\}=0.
\end{eqnarray*}
Thus $X_1,\ldots,X_n$ is a martingale difference sequence with respect
to the filtration $\emptyset\subset\mathcal{E}_1\subset\cdots
\subset\mathcal{E}_{n-1}$. In Girko \cite{Girko},
under the assumption of
the existence of $(4+\delta)$th moments of the matrix entries for some
$\delta>0$, Girko used the CLT for martingales to show the first term
of (\ref{2.3}) is asymptotically Gaussian. He also showed that the
last term of (\ref{2.3}) is asymptotically negligible. However, some
steps in the proofs of these two parts are lack of mathematical rigour.
Moreover, we do not find the discussion of the second term of (\ref
{2.3}) in Girko's original proof in \cite{Girko}.
Recently, Nguyen and
Vu provided a complete proof under the assumption that the
distributions of the matrix entries satisfy (\ref{1.2}), thus with
finite moments of all orders. In the following sections, we will also
adopt the representation (\ref{2.4}) and the theory on the weak
convergence of martingales.

It will be clear that the proof will rely on some approximations of
$X_{i+1}$ and $R_{i+1}$. However, these approximations are
$i$-dependent and the large $i$ case turns out to be badly
approximated. To see this, we can take the Gaussian case for example.
Note that when the entries are standard Gaussian, $\gamma_{i+1}^2\sim
\chi^2_{n-i}$, thus
\begin{eqnarray*}
X_{i+1}=\mathcal{O}\bigl((n-i)^{-1/2}\bigr),\qquad  R_{i+1}=
\mathcal{O}\bigl((n-i)^{-3/2}\bigr)
\end{eqnarray*}
with high probability. Especially, when $n-i$ is $\mathcal{O}(1)$, the
main term $X_{i+1}$ and the negligible term $R_{i+1}$ are comparable to
be $\mathcal{O}(1)$. Such a fact will be an obstacle if we use crude
estimations for $X_{i+1}$ and $R_{i+1}$ for general distribution case.
This is explained such as follows. When we estimate the last term of
(\ref{2.3}), a basic strategy is to use the Taylor expansion of $\log
(1+X_{i+1})$ to gain a relatively small remainder $R_{i+1}$, which
requires $|X_{i+1}|\leq1-c$ for some small positive constant $c$.
However, as we mentioned above, when $n-i$ is too small, such a bound
is hard to be guaranteed since $X_{i+1}=\mathcal{O}(1)$ with high
probability, especially under the assumption of the $4$th moment.
Fortunately, if all the $\mathbf{a}_{i+1}^T$'s are Gaussian for large
$i\geq n-s_1$ for some positive number $s_1$, $\gamma_{i+1}^2$ are
independent $\chi^2$ variables for all $i\geq n-s_1$ even if
$A_{(n-s_1)}$ is generally distributed. Such an explicit distribution
information can be used to deal with the large $i$ part. Therefore, in
 \cite{Girko} Girko proposed to replace some rows
by Gaussian ones and
prove the logarithmic law for the matrix after replacement, and then
recover the result to the original one by a comparison procedure. Such
a strategy was also used in Nguyen and Vu \cite{NV}. To pursue this
idea, we
set
\[
s_1=\bigl\lfloor\log^{3a} n\bigr\rfloor
\]
for some sufficiently large positive constant $a$. Our proof route can
be split into the following five steps.
\begin{longlist}[(iii)]
\item[(i)]
\begin{eqnarray*}
\frac{\sum_{i=0}^{n-s_1}X_{i+1}}{\sqrt{2\log n}}\stackrel {d}\longrightarrow N(0,1).
\end{eqnarray*}

\item[(ii)]
\begin{eqnarray*}
\frac{\sum_{i=0}^{n-s_1}X_{i+1}^2/2-\log n}{\sqrt{2\log n}}\stackrel {\mathbb{P}}\longrightarrow0.
\end{eqnarray*}

\item[(iii)]
\begin{eqnarray*}
\frac{\sum_{i=0}^{n-s_1}R_{i+1}}{\sqrt{2\log n}}\stackrel{\mathbb {P}}\longrightarrow0.
\end{eqnarray*}

\item[(iv)] If the last $s_1$ rows of $A_n$ are Gaussian, then
\begin{eqnarray*}
\frac{\sum_{i=n-s_1}^{n-1}\log\sklafrac{\gamma_{i+1}^2}{n-i}}{\sqrt {2\log n}}\stackrel{\mathbb{P}}\longrightarrow0.
\end{eqnarray*}

\item[(v)] Let $B_n$ be a random matrix satisfying the
basic Assumption \textup{\ref{assC}} and differing from $A_n$ only in the
last $s_1$ rows. Then one has
\begin{eqnarray*}
\sup_{x}\bigl|\mathbb{P}\bigl\{|\det A_n|\leq x\bigr\}-
\mathbb{P}\bigl\{|\det B_n|\leq x\bigr\}\bigr|\longrightarrow0.
\end{eqnarray*}
\end{longlist}

We will prove (i) and (ii) together in
Section~\ref{sec3}, and prove (iii) and (iv) in
Section~\ref{sec4}. Section~\ref{sec5} is devoted to the proof of (v).

\section{Convergence issue on the martingale difference sequence}\label{sec3}
In this section, we will prove the statements (i) and
(ii). The arguments for both two parts heavily rely on the
fact that $\{X_i,1\leq i\leq n\}$ is a martingale difference sequence.

In order to prove (i), we will use the following classical
CLT for martingales, which can be found in the book of
Hall and Heyde \cite{HH}, for instance.

\begin{pro}Let $\{S_{ni},\mathcal{F}_{ni},1\leq i\leq k_n,n\geq1\}$
be a zero-mean, square-integrable martingale array with differences
$Z_{ni}$. Suppose that
\begin{eqnarray*}
\max_{i}|Z_{ni}|&\stackrel{\mathbb{P}}
\longrightarrow&0,
\\
\sum_{i}Z_{ni}^2&\stackrel{
\mathbb{P}}\longrightarrow&1.
\end{eqnarray*}
Moreover, $\mathbb{E}(\max_{i}Z_{ni}^2)$ is bounded in $n$.
Then we have
\begin{eqnarray*}
S_{nk_n}\stackrel{d}\longrightarrow N(0,1).
\end{eqnarray*}
\end{pro}

Now we use the above proposition to prove (i). Let
$k_n=n-s_1$, $Z_{ni}=X_{i}/\sqrt{2\log n}$. Thus, it suffices to show that
%
\begin{eqnarray}\label{3.1}
\frac{1}{\sqrt{\log n}}\max_{0\leq i\leq n-s_1}|X_{i+1}|&\stackrel {
\mathbb{P}}\longrightarrow&0,
\\
\label{3.2}
\frac{1}{2\log n}\sum_{i=0}^{n-s_1}X_{i+1}^2
&\stackrel{\mathbb {P}}\longrightarrow&1,
\end{eqnarray}
and
%
\begin{eqnarray}\label{3.3}
\frac{1}{\log n}\mathbb{E}\Bigl(\max_{i}X_{i+1}^2
\Bigr)\leq\frac{1}{\log
n}\mathbb{E}\sum_{i=0}^{n-s_1}X_{i+1}^2
\leq C
\end{eqnarray}
for some positive constant $C$ independent of $n$.
To verify (\ref{3.1}) it suffices to show the following lemma.

\begin{lem} \label{lem.3.2.1}
Under the Assumption \textup{\ref{assC}}, we have for any constant
$\epsilon>0$
\begin{eqnarray*}
\sum_{i=0}^{n-s_1}\mathbb{P}\biggl\{
\frac{1}{\sqrt{\log n}}|X_{i+1}|\geq \epsilon\biggr\}\longrightarrow0
\end{eqnarray*}
as $n$ tends to infinity.
\end{lem}

\begin{pf} Below we use the notation
\begin{eqnarray*}
Q_{i}=\bigl[q_{jk}(i)\bigr]_{n,n}=:
\frac{1}{n-i}P_i.
\end{eqnarray*}
Thus by definition and the fact that $\tr Q_i=1$ we have
\begin{eqnarray*}
X_{i+1}=\mathbf{a}_{i+1}^TQ_i
\mathbf{a}_{i+1}-1=\sum_{k}q_{kk}(i)
\bigl(a_{i+1,k}^2-1\bigr)+\sum_{u\neq v}q_{uv}(i)a_{i+1,u}a_{i+1,v}.
\end{eqnarray*}
Now we introduce the quantities
\begin{eqnarray*}
U_{i+1}=\sum_{k}q_{kk}(i)
\bigl(a_{i+1,k}^2-1\bigr),\qquad  V_{i+1}=\sum
_{u\neq
v}q_{uv}(i)a_{i+1,u}a_{i+1,v}.
\end{eqnarray*}
Obviously
\begin{eqnarray*}
|X_{i+1}|\leq|U_{i+1}|+|V_{i+1}|.
\end{eqnarray*}
Then it is elementary to see
\begin{eqnarray*}
\mathbb{P}\biggl\{\frac{1}{\sqrt{\log n}}|X_{i+1}|\geq\epsilon\biggr\}\leq
\mathbb{P}\biggl\{\frac{1}{\sqrt{\log n}}|U_{i+1}|\geq\frac{\epsilon
}{2}\biggr
\}+\mathbb{P}\biggl\{\frac{1}{\sqrt{\log n}}|V_{i+1}|\geq\frac
{\epsilon}{2}
\biggr\}.
\end{eqnarray*}
Therefore, it suffices to verify the following two statements instead:
\begin{eqnarray*}
\sum_{i=0}^{n-s_1}\mathbb{P}\biggl\{
\frac{1}{\sqrt{\log n}}|U_{i+1}|\geq \frac{\epsilon}{2}\biggr\}
\longrightarrow0
\end{eqnarray*}
and
\begin{eqnarray*}
\sum_{i=0}^{n-s_1}\mathbb{P}\biggl\{
\frac{1}{\sqrt{\log n}}|V_{i+1}|\geq \frac{\epsilon}{2}\biggr\}
\longrightarrow0,
\end{eqnarray*}
which can be implied by
%
\begin{eqnarray}\label{3.4}
\frac{1}{\log n}\sum_{i=0}^{n-s_1}
\mathbb{E}U_{i+1}^2\longrightarrow 0,
\end{eqnarray}
and
%
\begin{eqnarray}\label{3.5}
\frac{1}{\log^2n}\sum_{i=0}^{n-s_1}\mathbb
{E}V_{i+1}^{4}\longrightarrow0.
\end{eqnarray}

First, we verify (\ref{3.4}). In the sequel, we set
\[
s_2=\bigl\lfloor n\log^{-20a}n\bigr\rfloor.
\]
By definition, we see
\begin{eqnarray*}
\mathbb{E}U_{i+1}^2=\mathbb{E}\biggl(\sum
_{k}q_{kk}(i) \bigl(a_{i+1,k}^2-1
\bigr)\biggr)^2\leq C\mathbb{E}\sum_{k}q_{kk}^2(i).
\end{eqnarray*}
Using the basic fact $0\leq p_{kk}(i)\leq1$ one has $0\leq
q_{kk}(i)\leq1/(n-i)$. Taking this fact and $\tr Q_i=1$ into account we have
%
\begin{eqnarray}\label{0107.2}
\frac{1}{\log n}\sum_{i=0}^{n-s_1}
\mathbb{E}U_{i+1}^2 &\leq& C\frac{1}{\log n}\sum
_{i=0}^{n-s_1}\mathbb{E}\sum
_{k}q_{kk}^2(i)
\nonumber
\\
&\leq&C\Biggl(\frac{1}{\log n}\sum_{i=0}^{n-s_2}
\frac{1}{n-i}+\frac
{1}{\log n}\sum_{i=n-s_2}^{n-s_1}
\mathbb{E}\max_{k}q_{kk}(i)\Biggr)
\nonumber
\\
&\leq&C\frac{1}{\log n}\sum_{i=n-s_2}^{n-s_1}
\mathbb{E}\max_{k}q_{kk}(i)+\mathcal{O}\biggl(
\frac{\log\log n}{\log n}\biggr)
\\
&\leq&C\frac{1}{\log n}\sum_{i=n-s_2}^{n-s_1}
\frac{1}{n-i}\mathbb {E}\max_{k}p_{kk}(i)+
\mathcal{O}\biggl(\frac{\log\log n}{\log
n}\biggr)
\nonumber
\\
&\leq&C\max_{n-s_2\leq i\leq n-s_1}\mathbb{E}\max_{k}p_{kk}(i)+
\mathcal{O}\biggl(\frac{\log\log n}{\log n}\biggr). \nonumber
\end{eqnarray}
Now we need the following crucial technical lemma which will be used
repeatedly in the sequel.

\begin{lem} \label{lem.0107} Under the above notation, we have
%
\begin{eqnarray}
\max_{n-s_2\leq i\leq n-1}\mathbb{E}\max_{k}p_{kk}(i)
\leq C\log ^{-8a}n \label{3.7}
\end{eqnarray}
for some positive constant $C$.
\end{lem}

The proof of Lemma~\ref{lem.0107} will be given later. Now we proceed
to the proof of (\ref{3.4}) by assuming Lemma~\ref{lem.0107}. Note
that (\ref{3.4}) follows from (\ref{0107.2}) and (\ref{3.7})
immediately. Hence, it remains to show (\ref{3.5}). We need the
following simple deviation lemma for the quadratic form, whose proof is
quite elementary and will be given in Appendix \hyperref[appA]{A}.

\begin{lem} \label{lem.3.1.1} Suppose $x_i,i=1,\ldots,n$ are
independent real random variables with common mean zero and variance
$1$. Moreover, we assume $\max_{i}\mathbb{E}|x_i|^l\leq\nu_l$. Let
$M_n=(m_{ij})_{n,n}$ be a nonnegative definite matrix which is
deterministic. Then we have
%
\begin{eqnarray}\label{3.1.1.2}
\mathbb{E}\Biggl|\sum_{i=1}^n
m_{ii}x_i^2-\tr M_n\Biggr|^4
\leq C\bigl( \nu_8\tr M_n^4+\bigl(
\nu_4\tr M_n^2\bigr)^2\bigr)
\end{eqnarray}
and
%
\begin{eqnarray}\label{3.1.1.3}
\mathbb{E}\biggl|\sum_{u\neq v}m_{uv}x_ux_v\biggr|^4
\leq C\nu_4^2\bigl(\tr M_n^2
\bigr)^2
\end{eqnarray}
for some positive constant $C$.
\end{lem}

Note that by (\ref{3.1.1.3}) one has
\begin{eqnarray*}
\mathbb{E}V_{i+1}^4\leq C\mathbb{E}\bigl(\tr Q_i^2
\bigr)^2=C\frac{1}{(n-i)^2},
\end{eqnarray*}
which implies that
%
\begin{eqnarray}
\frac{1}{4\log^2n}\sum_{i=0}^{n-s_1}
\mathbb{E}V_{i+1}^{4}=\mathcal {O}\bigl(\log^{-2-3a}n
\bigr).
\end{eqnarray}
Thus (\ref{3.5}) holds. Then Lemma~\ref{lem.3.2.1} follows from (\ref
{3.4}) and (\ref{3.5}) immediately. Thus (\ref{3.1}) is verified.
\end{pf}
Now we prove Lemma~\ref{lem.0107}.
\begin{pf*}{Proof of Lemma~\ref{lem.0107}}
We denote the $j$th column of $A_{(i)}$ by $\mathbf{b}_j(i)$ and use
the notation $A_{(i,j)}$ to denote the matrix induced from $A_{(i)}$ by
deleting the $j$th column $\mathbf{b}_j(i)$. Moreover, we set the
positive parameter $\alpha=\alpha_n:=n^{-1/6}$. By (\ref{2.5}), we have
\begin{eqnarray*}
p_{kk}(i)&=&1-\mathbf{b}_k(i)^T
\bigl(A_{(i)}A_{(i)}^T\bigr)^{-1}\mathbf
{b}_k(i)
\nonumber
\\
&=&1-\mathbf{b}_k(i)^T\bigl(A_{(i,k)}A_{(i,k)}^T+
\mathbf{b}_{k}(i)\mathbf {b}_k(i)^T
\bigr)^{-1}\mathbf{b}_{k}(i)
\nonumber
\\
&\leq&1-\mathbf{b}_k(i)^T\bigl(A_{(i,k)}A_{(i,k)}^T+n
\alpha I_i+\mathbf {b}_{k}(i)\mathbf{b}_k(i)^T
\bigr)^{-1}\mathbf{b}_{k}(i)
\nonumber
\\
&=& \bigl(1+\mathbf{b}_k(i)^T\bigl(A_{(i,k)}A_{(i,k)}^T+n
\alpha I_i \bigr)^{-1}\mathbf{b}_k(i)
\bigr)^{-1},
\end{eqnarray*}
where in the last step we used the Sherman--Morrison formula
%
\begin{eqnarray}\label{muji111}
\bigl(M+\mathbf{b}_k(i)\mathbf{b}_k(i)^T
\bigr)^{-1}=M^{-1}-\frac{M^{-1}\mathbf
{b}_k(i)\mathbf{b}_k(i)^T M^{-1}}{1+\mathbf{b}_k(i)^TM^{-1}\mathbf
{b}_k(i)}
\end{eqnarray}
for $k\times k$ invertible matrix $M$.
Let
\begin{eqnarray*}
G_{(i,k)}(\alpha)=\biggl(\frac{1}{n}A_{(i,k)}A_{(i,k)}^T+
\alpha I_i\biggr)^{-1}, \qquad G_{(i)}(\alpha)=\biggl(
\frac{1}{n}A_{(i)}A_{(i)}^T+\alpha
I_i\biggr)^{-1}.
\end{eqnarray*}
Then one has
\begin{eqnarray*}
p_{kk}(i)\leq \biggl(1+\frac{1}{n}\mathbf{b}_k(i)^TG_{(i,k)}(
\alpha )\mathbf{b}_k(i) \biggr)^{-1}.
\end{eqnarray*}
Hence, to verify (\ref{3.7}) we only need to show
%
\begin{eqnarray}\label{3.3.3.3}
\mathbb{E}\max_{k} \biggl(1+\frac{1}{n}\mathbf
{b}_k(i)^TG_{(i,k)}(\alpha)\mathbf{b}_k(i)
\biggr)^{-1}\leq C\log ^{-8a}n , \qquad n-s_2\leq i\leq
n-1.
\end{eqnarray}
It is apparent that $G_{(i)}(\alpha)$ and $G_{(i,k)}(\alpha)$ are
positive-definite and
\[
\bigl\llVert G_{(i)}(\alpha)\bigr\rrVert _{\mathrm{op}},\bigl\llVert
G_{(i,k)}(\alpha)\bigr\rrVert _{\mathrm{op}}\leq\alpha^{-1}.
\]
Moreover, we have
%
\begin{eqnarray}\label{3.1.1.1}
&&\hspace*{-15pt}\bigl|\tr G_{(i)}(\alpha)-\tr G_{(i,k)}(\alpha)\bigr|
\nonumber
\\
&&\hspace*{-15pt}\quad =\biggl|\tr\biggl(\frac{1}nA_{(i,k)}A_{(i,k)}^T+
\frac{1}n\mathbf{b}_k(i)\mathbf {b}_k(i)^T+
\alpha I_i\biggr)^{-1}-\tr\biggl(\frac{1}nA_{(i,k)}A_{(i,k)}^T+
\alpha I_i\biggr)^{-1}\biggr|
\\
&&\hspace*{-15pt}\quad =\frac{\sklfrac{1}n\mathbf{b}_k(i)^TG_{(i,k)}(\alpha)^2\mathbf
{b}_{k}(i)}{1+\sklfrac{1}n\mathbf{b}_k(i)^TG_{(i,k)}(\alpha)\mathbf
{b}_{k}(i)}\leq\alpha^{-1},\nonumber
\end{eqnarray}
where in the second step above we used the Sherman--Morrison formula
(\ref{muji111}) again.
Now we set
\begin{eqnarray*}
\chi(i)=\mathbf{1}_{\{\sklfrac{1}n\tr G_{(i)}(\alpha)\geq\log^{10a}n\}},
\end{eqnarray*}
and we denote the $(u,v)$th entry of $G_{(i,k)}(\alpha)$ by
$G_{(i,k)}(u,v)$ below. Moreover, for ease of presentation, when there
is no confusion, we will omit the parameter $\alpha$ from the notation
$G_{(i,k)}(\alpha)$ and $G_{(i)}(\alpha)$.
Then we have for some small constant $0<\epsilon<1/2$,
%
\begin{eqnarray}\label{3.6}
&&\mathbb{E}\max_{k}\biggl(1+\frac{1}{n}
\mathbf{b}_k(i)^TG_{(i,k)}\mathbf
{b}_k(i)\biggr)^{-1}
\nonumber
\\
&&\quad \leq\mathbb{P}\biggl\{\frac{1}n\tr G_{(i)}\leq
\log^{10a}n\biggr\}+\mathbb {E}\chi(i)\max_{k}
\biggl(1+\frac{1}{n}\mathbf{b}_k(i)^TG_{(i,k)}
\mathbf {b}_k(i)\biggr)^{-1}
\nonumber
\\
&&\quad \leq \mathbb{P}\biggl\{\frac{1}n\tr G_{(i)}\leq
\log^{10a}n\biggr\}+\mathbb {E}\chi(i)\max_{k}
\Biggl(1+\frac{1}{n}\sum_{j=1}^iG_{(i,k)}(j,j)a_{jk}^2-
\epsilon\Biggr)^{-1}
\nonumber
\\
&&\qquad {}+\mathbb{P}\biggl\{\frac{1}n\max_{1\leq k\leq n}\biggl|\sum
_{1\leq u\neq v\leq
i}G_{(i,k)}(u,v)a_{uk}a_{vk}\biggr|
\geq\epsilon\biggr\}
\nonumber
\\
&&\quad \leq\mathbb{P}\biggl\{\frac{1}n\tr G_{(i)}\leq
\log^{10a}n\biggr\}+\mathbb {E}\chi(i)\max_k
\biggl(1+\log^{-a}n\cdot\frac{1}{n}\tr G_{(i,k)}-\epsilon
\biggr)^{-1}\nonumber
\\[-8pt]\\[-8pt]
&&\qquad {}+C\sum_{k=1}^n\mathbb{P}\Biggl\{\sum
_{j=1}^iG_{(i,k)}(j,j)a_{jk}^2<
\log ^{-a}n\cdot \tr G_{(i,k)},\frac{1}n\tr
G_{(i)}\geq\log^{10a}n\Biggr\}
\nonumber
\\[-1pt]
&&\qquad {}+\mathbb{P}\biggl\{\frac{1}n\max_{1\leq k\leq n}\biggl|\sum
_{1\leq u\neq v\leq
i}G_{(i,k)}(u,v)a_{uk}a_{vk}\biggr|
\geq\epsilon\biggr\}
\nonumber
\\[-1pt]
&&\quad \leq\mathbb{P}\biggl\{\frac{1}n\tr G_{(i)}\leq
\log^{10a}n\biggr\}+\mathbb {E}\chi(i) \biggl(1+\log^{-a}n\cdot
\frac{1}{n}\tr G_{(i)}-2\epsilon \biggr)^{-1}
\nonumber
\\[-1pt]
&&\qquad {}+C\sum_{k=1}^n\mathbb{P}\Biggl\{\sum
_{j=1}^iG_{(i,k)}(j,j)a_{jk}^2<
\log^{-a}n \tr G_{(i,k)},\frac{1}n\tr G_{(i)}\geq
\log^{10a}n\Biggr\}
\nonumber
\\[-1pt]
&&\qquad {}+\mathbb{P}\biggl\{\frac{1}n\max_{1\leq k\leq n}\biggl|\sum
_{1\leq u\neq v\leq
i}G_{(i,k)}(u,v)a_{uk}a_{vk}\biggr|
\geq\epsilon\biggr\},\nonumber
\end{eqnarray}
where in the above last inequality, we used (\ref{3.1.1.1}).
Below we will estimate (\ref{3.6}) term by term. To this end, we need
the following two lemmas, whose proofs will be given in Appendix \hyperref[appA]{A}.

\begin{lem} \label{lem.3.1} Under the assumption of Theorem~\ref
{thm.1.1}, for $n-s_2\leq i\leq n-1$, we have for $\alpha=n^{-1/6}$\vspace*{-1pt}
\begin{eqnarray*}
\mathbb{E}\biggl\{\frac{1}n\tr G_{(i)}(\alpha)\biggr
\}=s_i(\alpha)+\mathcal{O}\bigl(n^{-1/6}\bigr)
\end{eqnarray*}
and\vspace*{-1pt}
\begin{eqnarray*}
\Var\biggl\{\frac{1}n\tr G_{(i)}(\alpha)\biggr\}=\mathcal{O}
\bigl(n^{-1/3}\bigr),
\end{eqnarray*}
where\vspace*{-1pt}
\begin{eqnarray*}
s_i(\alpha)=2 \biggl(\alpha+1-\frac{i}{n}+\sqrt{\biggl[
\alpha+\biggl(1-\frac
{i}{n}\biggr)\biggr]^2+4\alpha
\frac{i}{n}} \biggr)^{-1}.
\end{eqnarray*}
\end{lem}

With the aid of Lemma~\ref{lem.3.1}, we can estimate the first term of
(\ref{3.6}) as follows. Note that by definition $s_i(\alpha)\geq
\frac{1}{10}\log^{20a}n$ for $n-s_2\leq i\leq n-1$, we have\vspace*{-1pt}
\begin{eqnarray*}
\mathbb{P}\biggl\{\frac{1}n\tr G_{(i)}\leq\log^{10a}n
\biggr\} &\leq& \mathbb{P}\biggl\{\biggl|\frac{1}n\tr G_{(i)}-
\mathbb{E}\frac{1}n\tr G_{(i)}\biggr|\geq\frac{1}{20}
\log^{20a} n\biggr\}
\nonumber
\\
&\leq&C\log^{-40a}n\Var\biggl\{\frac{1}n\tr G_{(i)}
\biggr\}=\mathrm{o}\bigl(n^{-1/3}\bigr).
\end{eqnarray*}
For the second term of (\ref{3.6}), with the definition of $\chi(i)$,
obviously one has
\begin{eqnarray*}
\mathbb{E}\chi(i) \biggl(1+\log^{-a}n\cdot\frac{1}{n}\tr G_{(i)}-2
\epsilon \biggr)^{-1}\leq C\log^{-8a}n.
\end{eqnarray*}
Now we deal with the third term of (\ref{3.6}). We set
\begin{eqnarray*}
\hat{a}_{jk}=a_{jk}\mathbf{1}_{\{|a_{jk}|\leq\log^{a}n\}},\qquad  \tilde
{a}_{jk}=\frac{\hat{a}_{jk}-\mathbb{E}\hat{a}_{jk}}{\sqrt{\Var\{
\hat{a}_{jk}\}}}.
\end{eqnarray*}
Since $G_{(i,k)}(j,j)$'s are positive and $\hat{a}_{jk}^2\leq
a_{jk}^2$ one has
\begin{eqnarray*}
&&\sum_{k=1}^n\mathbb{P}\Biggl\{\sum
_{j=1}^iG_{(i,k)}(j,j)a_{jk}^2<
\log ^{-a}n\cdot \tr G_{(i,k)},\frac{1}n\tr
G_{(i)}\geq\log^{10a}n\Biggr\}
\nonumber
\\
&&\quad \leq\sum_{k=1}^n\mathbb{P}\biggl\{\sum
_{j}G_{(i,k)}(j,j)\hat {a}_{jk}^2<
\log^{-a}n\cdot \tr G_{(i,k)},\frac{1}n\tr
G_{(i)}\geq\log^{10a}n\biggr\}.
\end{eqnarray*}
Moreover, by the assumption $\sup_n\max_{ij}\mathbb
{E}a_{ij}^4<\infty$ it is easy to derive that\vspace*{-1pt}
\begin{eqnarray*}
\mathbb{E}\hat{a}_{jk}=\mathcal{O}\bigl(\log^{-3a}n\bigr),\qquad
\Var\{\hat {a}_{jk}\}=1+\mathcal{O}\bigl(\log^{-2a}n\bigr).
\end{eqnarray*}
Consequently,\vspace*{-1pt}
\begin{eqnarray*}
\tilde{a}_{jk}=\hat{a}_{jk}+\mathcal{O}\bigl(
\log^{-a}n\bigr),
\end{eqnarray*}
which implies\vspace*{-1pt}
\begin{eqnarray*}
\tilde{a}_{jk}^2\leq2\hat{a}_{jk}^2+
\mathcal{O}\bigl(\log^{-2a}n\bigr)\leq 2\hat{a}_{jk}^2+
\log^{-a}n
\end{eqnarray*}
for sufficiently large $n$.
Therefore, we have
\begin{eqnarray*}
&&\sum_{k=1}^n\mathbb{P}\biggl\{\sum
_{j}G_{(i,k)}(j,j)a_{jk}^2<
\log ^{-a}n\cdot \tr G_{(i,k)},\frac{1}n\tr
G_{(i)}\geq\log^{10a}n\biggr\}
\nonumber
\\
&&\quad \leq\sum_{k=1}^n\mathbb{P}\biggl\{\sum
_{j}G_{(i,k)}(j,j)\hat {a}_{jk}^2<
\log^{-a}n\cdot \tr G_{(i,k)},\frac{1}n\tr
G_{(i)}\geq\log^{10a}n\biggr\}
\nonumber
\\
&&\quad \leq\sum_{k=1}^n\mathbb{P}\biggl\{\sum
_{j}G_{(i,k)}(j,j)\tilde
{a}_{jk}^2<3\log^{-a}n\cdot \tr G_{(i,k)},
\frac{1}n\tr G_{(i)}\geq\log^{10a}n\biggr\}
\nonumber
\\
&&\quad \leq\sum_{k=1}^n\mathbb{P}\biggl\{\biggl|
\sum_{j}G_{(i,k)}(j,j)\tilde{a}_{jk}^2-
\tr G_{(i,k)}\biggr|\geq\frac{1}2\tr G_{(i,k)},
\frac{1}n\tr G_{(i)}\geq\log ^{10a}n\biggr\}
\nonumber
\\
&&\quad \leq\sum_{k=1}^n\mathbb{P}\biggl\{
\frac{1}n\biggl|\sum_{j}G_{(i,k)}(j,j)
\tilde {a}_{jk}^2- \tr G_{(i,k)}\biggr|\geq
\frac{1}4\log^{10a}n\biggr\}
\nonumber
\\
&&\quad \leq C\log^{-40a}n\cdot n^{-4}\sum
_{k=1}^n\mathbb{E}\bigl[\log ^{4a}n\cdot
\tr G_{(i,k)}^{4}+\bigl(\tr G_{i,k}^2
\bigr)^{2}\bigr]
\nonumber
\\
&&\quad \leq\mathrm{o}\biggl(\frac{1}{n \alpha^{4}}\biggr).
\end{eqnarray*}
In the fourth inequality we used the fact (\ref{3.1.1.1}) and in the
fifth inequality we used (\ref{3.1.1.2}) and the fact $\mathbb
{E}|\tilde{a}_{ij}|^{4+t}=\mathcal{O}(\log^{ta}n)$ for any $t\geq
0$, which is easy to see from the definition of $\tilde{a}_{ij}$.
Now we begin to deal with the last term of (\ref{3.6}). Note that by
(\ref{3.1.1.3})
\begin{eqnarray*}
&&\mathbb{P}\biggl\{\frac{1}n\max_{1\leq k\leq n}\biggl|\sum
_{u\neq
v}G_{(i,k)}(u,v)a_{uk}a_{vk}\biggr|
\geq\epsilon\biggr\}
\nonumber
\\
&&\quad \leq\sum_{k=1}^n\mathbb{P}\biggl\{
\frac{1}n\biggl|\sum_{u\neq
v}G_{(i,k)}(u,v)a_{uk}a_{vk}\biggr|
\geq\epsilon\biggr\}
\nonumber
\\
&&\quad \leq\epsilon^{-4}n^{-4}\sum_{k=1}^n
\mathbb{E}\biggl(\sum_{u\neq
v}G_{(i,k)}(u,v)a_{uk}a_{vk}
\biggr)^4
\nonumber
\\
&&\quad \leq C\epsilon^{-4}n^{-4}\sum_{k=1}^n
\mathbb{E}\bigl(\tr G_{(i,k)}^2\bigr)^2
\nonumber
\\
&&\quad \leq\mathcal{O}\biggl(\frac{1}{n\alpha^4}\biggr).
\end{eqnarray*}
Therefore, (\ref{3.3.3.3}) follows from the above estimates, so does
(\ref{3.7}). Hence, we complete the proof.
\end{pf*}

Now we come to deal with (\ref{3.2}) and (\ref{3.3}). Note that (\ref
{3.2}) can be implied by (ii) directly. Thus we will prove
the statement (ii) and (\ref{3.3}) below. We reformulate
them as the following lemma and then prove it.

\begin{lem} \label{lem.3.1.1.1}Under the Assumption \textup{\ref{assC}},
one has
%
\begin{eqnarray}\label{3.1.1}
\frac{\sum_{i=0}^{n-s_1}X_{i+1}^2-2\log n}{\sqrt{\log n}}\stackrel {\mathbb{P}} \longrightarrow0,
\end{eqnarray}
and
%
\begin{eqnarray}\label{3.1.2}
\frac{1}{2\log n}\sum_{i=0}^{n-s_1}
\mathbb{E}X_{i+1}^2=1+\mathrm {o}(1).
\end{eqnarray}
\end{lem}

\begin{pf} We begin with (\ref{3.1.1}). We split the proof of (\ref
{3.1.1}) into two steps.
%
\begin{eqnarray}\label{3.8}
\frac{\sum_{i=0}^{n-s_1}(X_{i+1}^2-\mathbb{E}\{X_{i+1}^2|\mathcal
{E}_i\})}{\sqrt{\log n}}\stackrel{\mathbb{P}}\longrightarrow0,
\end{eqnarray}
and
%
\begin{eqnarray}\label{3.9}
\frac{\sum_{i=0}^{n-s_1}\mathbb{E}\{X_{i+1}^2|\mathcal{E}_i\}-2\log
n}{\sqrt{\log n}}\stackrel{\mathbb{P}}\longrightarrow0.
\end{eqnarray}
Observe that
%
\begin{eqnarray}\label{3.12}
\mathbb{E}\bigl\{X_{i+1}^2|\mathcal{E}_{i}\bigr
\} &=&\mathbb{E}\Biggl\{\Biggl(\sum_{j,k=1}^nq_{jk}(i)a_{i+1,j}a_{i+1,k}-1
\Biggr)^2\Bigl|\mathcal{E}_i\Biggr\}
\nonumber
\\[-8pt]\\[-8pt]
&=&\frac{2}{n-i}+\sum_{k}q_{kk}(i)^2
\bigl(\mathbb{E}|a_{i+1,k}|^4-3\bigr). \nonumber
\end{eqnarray}
Therefore, to verify (\ref{3.9}), we only need
%
\begin{eqnarray}\label{3.13}
\frac{1}{\sqrt{\log n}}\mathbb{E}\sum_{i=0}^{n-s_1}
\sum_{k}q_{kk}(i)^2
\longrightarrow0.
\end{eqnarray}
Note that from the proof of (\ref{3.4}) we can get the following
estimate directly.
%
\begin{eqnarray}\label{3.3.3.3.3}
\frac{1}{\sqrt{\log n}}\mathbb{E}\sum_{i=0}^{n-s_1}
\sum_{k}q_{kk}(i)^2=
\mathcal{O}\biggl(\frac{\log\log n}{\sqrt{\log n}}\biggr).
\end{eqnarray}
Thus it suffices to show (\ref{3.8}). By elementary calculations, we have
\begin{eqnarray*}
&&X_{i+1}^2-\mathbb{E}\bigl\{X_{i+1}^2|
\mathcal{E}_i\bigr\}
\nonumber
\\
&&\quad =-2\sum_{u}q_{uu}(i)
\bigl(a_{i+1,u}^2-1\bigr)+2\sum_{u\neq
v}q_{uu}(i)q_{vv}(i)
\bigl(a_{i+1,u}^2a_{i+1,v}^2-1\bigr)
\nonumber
\\
&&\qquad {}+2\sum_{u\neq v}q_{uv}(i)^2
\bigl(a_{i+1,u}^2a_{i+1,v}^2-1\bigr)
\nonumber
\\
&&\qquad {}+2\mathop{\sum_{u_1\neq v_1,u_2\neq v_2}}_{\{u_1,v_1\}\neq\{
u_2\neq v_2\}
}q_{u_1v_1}(i)q_{u_2v_2}(i)a_{i+1,u_1}a_{i+1,v_1}a_{i+1,u_2}a_{i+1,v_2}
\nonumber
\\
&&\qquad {}+\sum_{u}q_{uu}(i)^2
\bigl(a_{i+1,u}^4-\mathbb{E}a_{i+1,u}^4
\bigr)+2\biggl(\sum_{u}q_{uu}(i)
\bigl(a_{i+1,u}^2-1\bigr)\biggr) \biggl(\sum
_{u\neq
v}q_{uv}(i)a_{i+1,u}a_{i+1,v}
\biggr)
\nonumber
\\
&&\quad =:2W_1(i)+2W_2(i),
\end{eqnarray*}
where\vspace*{2pt}
\begin{eqnarray*}
W_1(i)&=&-\sum_{i}q_{uu}(i)
\bigl(a_{i+1,u}^2-1\bigr)+\sum_{u\neq
v}q_{uu}(i)q_{vv}(i)
\bigl(a_{i+1,u}^2a_{i+1,v}^2-1\bigr)
\nonumber
\\[1.5pt]
&&{}+\sum_{u\neq v}q_{uv}(i)^2
\bigl(a_{i+1,u}^2a_{i+1,v}^2-1\bigr)
\nonumber
\\[1.5pt]
&&{}+\mathop{\sum_{u_1\neq v_1,u_2\neq v_2}}_{\{u_1,v_1\}\neq\{u_2\neq
v_2\}}q_{u_1v_1}(i)q_{u_2v_2}(i)a_{i+1,u_1}a_{i+1,v_1}a_{i+1,u_2}a_{i+1,v_2},
\end{eqnarray*}
and\vspace*{2pt}
\begin{eqnarray*}
W_2(i)&=&\frac{1}2\sum_{u}q_{uu}(i)^2
\bigl(a_{i+1,u}^4-\mathbb {E}a_{i+1,u}^4
\bigr)
\nonumber
\\[1.5pt]
&&{}+\biggl(\sum_{u}q_{uu}(i)
\bigl(a_{i+1,u}^2-1\bigr)\biggr) \biggl(\sum
_{u\neq
v}q_{uv}(i)a_{i+1,u}a_{i+1,v}
\biggr).
\end{eqnarray*}
We split the issue to show\vspace*{2pt}
%
\begin{eqnarray}\label{3.10}
\frac{1}{\sqrt{\log n}}\sum_{i=0}^{s_1}W_1(i)
\stackrel{\mathbb {P}}\longrightarrow0,\qquad  \frac{1}{\sqrt{\log n}}\sum
_{i=0}^{s_1}W_2(i)\stackrel{\mathbb{P}}
\longrightarrow0.
\end{eqnarray}
First, we deal with the second statement of (\ref{3.10}). Note that\vspace*{2pt}
%
\begin{eqnarray}\label{3.11}
&&\frac{1}{\sqrt{\log n}}\sum_{i=0}^{n-s_1}
\mathbb{E}\bigl|W_2(i)\bigr|\nonumber\\[1.5pt]
&&\quad \leq C\frac{1}{\sqrt{\log n}}\sum
_{i=0}^{n-s_1}\mathbb{E}\sum
_{u}q_{uu}(i)^2\nonumber
\\[-7pt]\\[-7pt]
&&\qquad {}+C\frac{1}{\sqrt{\log n}}\sum_{i=0}^{n-s_1}
\biggl(\mathbb{E}\biggl(\sum_{u}q_{uu}(i)
\bigl(a_{i+1,u}^2-1\bigr)\biggr)^2
\biggr)^{1/2}\biggl(\mathbb{E}\biggl(\sum_{u\neq
v}q_{uv}(i)a_{i+1,u}a_{i+1,v}
\biggr)^2\biggr)^{1/2}
\nonumber
\\[1.5pt]
&&\quad \leq C\frac{1}{\sqrt{\log n}}\sum_{i=0}^{n-s_1}
\mathbb{E}\sum_{u}q_{uu}(i)^2+C
\frac{1}{\sqrt{\log n}}\sum_{i=0}^{n-s_1}\biggl(
\frac
{1}{n-i}\biggr)^{1/2}\biggl(\mathbb{E}\sum
_{u}q_{uu}(i)^2\biggr)^{1/2}.
\nonumber
\end{eqnarray}
By (\ref{3.3.3.3.3}), we have that the first term of (\ref{3.11}) is
of the order of $\mathcal{O}(\log\log n/\sqrt{\log n})$. For the
second term, by using (\ref{3.7}) we have
\begin{eqnarray*}
&&\frac{1}{\sqrt{\log n}}\sum_{i=0}^{n-s_1}\biggl(
\frac
{1}{n-i}\biggr)^{1/2}\biggl(\mathbb{E}\sum
_{u}q_{uu}(i)^2\biggr)^{1/2}
\nonumber
\\
&&\quad \leq\frac{1}{\sqrt{\log n}}\sum_{i=0}^{n-s_2}
\frac{1}{n-i}+\frac
{1}{\sqrt{\log n}}\sum_{i=n-s_2}^{n-s_1}
\biggl(\frac{1}{n-i}\mathbb {E}\sum_{u}q_{uu}(i)^2
\biggr)^{1/2}
\nonumber
\\
&&\quad \leq\frac{1}{\sqrt{\log n}}\sum_{i=n-s_2}^{n-s_1}
\frac
{1}{n-i}\Bigl(\mathbb{E}\max_u
p_{uu}(i)\Bigr)^{1/2}+\mathcal{O}\biggl(\frac{\log
\log n}{\sqrt{\log n}}
\biggr)
\nonumber
\\
&&\quad =\mathcal{O}\biggl(\frac{\log\log n}{\sqrt{\log n}}\biggr).
\end{eqnarray*}
Now we consider the first term of (\ref{3.10}). It is easy to see
\begin{eqnarray*}
\mathbb{E}\bigl\{W_1(i)\bigr\}=0,\qquad  \mathbb{E}\bigl
\{W_1(i)W_1(j)\bigr\}=0, \qquad i\neq j,
\end{eqnarray*}
which yields
%
\begin{eqnarray}\label{3.2.2.2}
\mathbb{E}\Biggl(\frac{1}{\sqrt{\log n}}\sum_{i=0}^{s_1}W_1(i)
\Biggr)^2&=&\frac
{1}{\log n}\sum_{i=0}^{s_1}
\mathbb{E}W_1(i)^2
\nonumber
\\
&\leq&\frac{C}{\log n}\sum_{i=0}^{n-s_1}
\mathbb{E}\sum_{u}q_{uu}(i)^2+
\frac{C}{\log n}\sum_{i=0}^{n-s_1}\mathbb{E}
\sum_{u\neq v, u\neq w}q_{uu}(i)^2q_{vv}(i)q_{ww}(i)
\nonumber
\\[-8pt]\\[-8pt]
&&{}+\frac{C}{\log n}\sum_{i=0}^{n-s_1}
\mathbb{E}\sum_{u\neq v, u\neq
w}q_{uv}(i)^2q_{uw}(i)^2
\nonumber
\\
&&{}+\frac{C}{\log n}\sum_{i=0}^{n-s_1}
\mathbb{E}\sum_{u_1\neq
v_1,u_2\neq v_2}\bigl|q_{u_1v_1}(i)q_{u_1v_2}(i)q_{u_2v_1}(i)q_{u_2v_2}(i)\bigr|.
\nonumber
\end{eqnarray}
The estimation of (\ref{3.2.2.2}) is elementary but somewhat tedious.
In fact, one can find the estimate towards every term of (\ref
{3.2.2.2}) in Nguyen and Vu \cite{NV} (see the estimation of $\Var
(\sum_{i=0}^{n-s_1}Y_{i+1})$ in Section~6 of Nguyen and Vu \cite{NV}).
Here we omit the
details and claim the following estimation
\begin{eqnarray*}
\mathbb{E}\Biggl(\frac{1}{\sqrt{\log n}}\sum_{i=0}^{s_1}W_1(i)
\Biggr)^2=\mathcal{O}\biggl(\frac{\log\log n}{\log n}\biggr).
\end{eqnarray*}
Then (\ref{3.8}) follows, so does (\ref{3.1.1}). Moreover, it is easy
to see that (\ref{3.1.2}) holds by combining (\ref{3.12}) and (\ref
{3.13}). Therefore, Lemma~\ref{lem.3.1.1.1} is proved.
\end{pf}
Thus we complete the proof of (i) and (ii).

\section{Negligible parts (iii) and (iv)}\label{sec4}
In this section, we will prove the statements (iii) and
(iv).
We start with (iii). The following elementary but crucial
lemma will be needed.

\begin{lem} \label{lem.3.3}By the definitions above, if $X_{i+1}\geq
-1+\log^{-\sfrac{a}{2}}n$, one has
\begin{eqnarray*}
|R_{i+1}|\leq C\bigl(U_{i+1}^2+|V_{i+1}|^{2+\delta}
\bigr)\log\log n
\end{eqnarray*}
for any $0\leq\delta\leq1$. Here $C:=C(a,\delta)$ is a positive
constant only depends on $a$ and $\delta$.
\end{lem}

\begin{pf}
We split the discussion into three cases. Choose some small constant
$0<\epsilon<\frac{1}{10}$ (say) and consider the three cases
$|X_{i+1}|\leq1-\epsilon$, $X_{i+1}>1-\epsilon$ and $-1+\log
^{-\sfrac{a}{2}}n\leq X_{i+1}<-1+\epsilon$ separately. For the first
case, we can use the elementary Taylor expansion to see that
\begin{eqnarray*}
|R_{i+1}|\leq C|X_{i+1}|^3\leq
C|U_{i+1}+V_{i+1}|^{2+\delta}
\end{eqnarray*}
for any $0\leq\delta\leq1$. If $|U_{i+1}|\geq1$ or $|V_{i+1}|\geq
1$, we will immediately get
\begin{eqnarray*}
|R_{i+1}|\leq C \bigl(U_{i+1}^2+|V_{i+1}|^{2+\delta}
\bigr)
\end{eqnarray*}
since $|U_{i+1}+V_{i+1}|<1$. If both $|U_{i+1}|$ and $|V_{i+1}|$ are
less than $1$, we have
\begin{eqnarray*}
|R_{i+1}|\leq C\bigl(|U_{i+1}|^{2+\delta}+|V_{i+1}|^{2+\delta}
\bigr)\leq C\bigl(U_{i+1}^{2}+|V_{i+1}|^{2+\delta}
\bigr).
\end{eqnarray*}
Now we come to deal with the second case. When $X_{i+1}>1-\epsilon$,
obviously one has
\begin{eqnarray*}
|R_{i+1}|\leq C(U_{i+1}+V_{i+1})^2.
\end{eqnarray*}
Then it is elementary to see that we always can find some positive
constant $C$ such that
\begin{eqnarray*}
|R_{i+1}|\leq C\bigl(U_{i+1}^2+|V_{i+1}|^{2+\delta}
\bigr)
\end{eqnarray*}
since $\max\{U_{i+1},V_{i+1}\}>\frac{1}2-\frac{\epsilon}{2}$.
Finally, we deal with the last case. Note that when $-1+\log^{-c}n\leq
X_{i+1}<-1+\epsilon$, we have
\begin{eqnarray*}
|R_{i+1}|\leq C\log\log n.
\end{eqnarray*}
Moreover, it is obvious that we have $\max\{|U_{i+1}|,|V_{i+1}|\}>
\frac{1}2-\frac{\epsilon}{2}$. Consequently, one has
\begin{eqnarray*}
|R_{i+1}|\leq C\bigl(U_{i+1}^2+|V_{i+1}|^{2+\delta}
\bigr)\log\log n.
\end{eqnarray*}
In conclusion, we completed the proof.
\end{pf}
The next lemma is devoted to bounding the probability of the event
$\bigcup_{i=0}^{n-s_1}\{X_{i+1}< -1+\log^{-a}n\}$.

\begin{lem}\label{lem.3.4} Under the Assumption \textup{\ref{assC}}, we have
\begin{eqnarray*}
\sum_{i=0}^{n-s_1}\mathbb{P}\bigl
\{X_{i+1}< -1+\log^{-\sfrac{a}{2}}n\bigr\} \longrightarrow0
\end{eqnarray*}
as $n$ tends to infinity.
\end{lem}

\begin{pf}
Note that
\begin{eqnarray*}
&&\mathbb{P}\bigl\{X_{i+1}< -1+\log^{-\sfrac{a}{2}}n\bigr\}\\
&&\quad =\mathbb{P}
\bigl\{ \mathbf{a}_{i+1}^TQ_i
\mathbf{a}_{i+1}< \log^{-\sfrac{a}{2}}n\bigr\}
\nonumber
\\
&&\quad \leq\mathbb{P}\biggl\{\sum_{k}q_{kk}(i)a_{i+1,k}^2<2
\log^{-\sfrac
{a}{2}}n\biggr\}+\mathbb{P}\biggl\{\biggl|\sum_{u\neq
v}q_{uv}(i)a_{i+1,u}a_{i+1,v}\biggr|
\geq\frac{1}2\log^{-\sfrac{a}{2}}n\biggr\}.
\end{eqnarray*}
Now we recall the definition
\begin{eqnarray*}
\hat{a}_{i+1,k}=a_{i+1,k}\mathbf{1}_{\{|a_{i+1,k}|\leq\log^{a}n\}},\qquad
\tilde{a}_{i+1,k}=\frac{\hat{a}_{i+1,k}-\mathbb{E}\hat
{a}_{i+1,k}}{\sqrt{\Var\{\hat{a}_{i+1,k}\}}}.
\end{eqnarray*}
A similar discussion as that in the last section yields
\begin{eqnarray*}
\mathbb{P}\biggl\{\sum_{k}q_{kk}(i)a_{i+1,k}^2<2
\log^{-\sfrac{a}{2}}n\biggr\} &\leq& \mathbb{P}\biggl\{\sum
_{k}q_{kk}(i)\tilde{a}_{i+1,k}^2<C
\log ^{-\sfrac{a}{2}}n\biggr\}
\nonumber
\\
&\leq&\mathbb{P}\biggl\{\biggl|\sum_{k}q_{kk}(i)
\tilde{a}_{i+1,k}^2-1\biggr|\geq \frac{1}2\biggr\}
\nonumber
\\
&\leq&C\bigl(\log^{4a}n \cdot\tr Q_i^4+\bigl(
\tr Q_i^2\bigr)^2\bigr)
\nonumber
\\
&\leq&C\bigl(\log^{4a}n\cdot(n-i)^{-3}+(n-i)^{-2}
\bigr).
\end{eqnarray*}
Here in the third inequality, we used (\ref{3.1.1.2}) again. Moreover,
by (\ref{3.1.1.3}), we obtain
\begin{eqnarray*}
\mathbb{P}\biggl\{\biggl|\sum_{u\neq v}q_{uv}(i)a_{i+1,u}a_{i+1,v}\biggr|
\geq\frac{1}2\log^{-\sfrac{a}{2}}n\biggr\}\leq C(n-i)^{-2}
\log^{2a}n.
\end{eqnarray*}
Thus finally, we have
\begin{eqnarray*}
&&\sum_{i=0}^{n-s_1}\mathbb{P}\bigl
\{X_{i+1}< -1+\log^{-\sfrac{a}{2}}n\bigr\}
\nonumber
\\
&&\quad \leq C\Biggl(\log^{4a}n\sum_{i=0}^{n-s_1}(n-i)^{-3}+
\log^{2a}n\sum_{i=0}^{n-s_1}(n-i)^{-2}
\Biggr)
\nonumber
\\
&&\quad \leq C\log^{-a} n.
\end{eqnarray*}
Therefore, we complete the proof.
\end{pf}

Combining Lemmas \ref{lem.3.3} with \ref{lem.3.4}, one has with
probability $1-\mathrm{o}(1)$,
\begin{eqnarray*}
|R_{i+1}|\leq C\bigl(U_{i+1}^2+V_{i+1}^{2+\delta}
\bigr)\log\log n,\qquad  0\leq i\leq n-s_1.
\end{eqnarray*}
Thus to show (iii), it suffices to verify that
\begin{eqnarray*}
\frac{\log\log n}{\sqrt{\log n}}\sum_{i=0}^{n-s_1}
\bigl(U_{i+1}^2+|V_{i+1}|^{2+\delta}\bigr)
\stackrel{\mathbb {P}}\longrightarrow0,
\end{eqnarray*}
which can be implied by
%
\begin{eqnarray}\label{3.14}
\frac{\log\log n}{\sqrt{\log n}}\sum_{i=0}^{n-s_1}\mathbb
{E}\bigl(U_{i+1}^2+|V_{i+1}|^{2+\delta}\bigr)
\longrightarrow0.
\end{eqnarray}
Note that
%
\begin{eqnarray}\label{3.15}
\sum_{i=0}^{n-s_1}\mathbb{E}U_{i+1}^2
\leq C\sum_{i=0}^{n-s_1}\mathbb {E}\sum
_{j}q_{jj}(i)^2=\mathcal{O}(\log\log
n).
\end{eqnarray}
Moreover, we have
%
\begin{eqnarray}
\sum_{i=0}^{n-s_1}\mathbb{E}|V_{i+1}|^{2+\delta}
\leq\sum_{i=0}^{n-s_1}\bigl(
\mathbb{E}V_{i+1}^{4}\bigr)^{\vfrac{2+\delta}{4}}\leq C\sum
_{i=0}^{n-s_1}(n-i)^{-1-\delta/2}=\mathrm{o}(1).\label{3.16}
\end{eqnarray}
Then (\ref{3.14}) follows from (\ref{3.15}) and (\ref{3.16})
immediately. Thus, we completed the proof of (iii).

It remains to show (iv) in this section. The proof is quite
elementary owing to the fact that $\{\gamma_{i+1}^2,i=n-s_1,\ldots
,n-1\}$ is an independent sequence and $\gamma_{i+1}^2\sim\chi
^2_{n-i}$. One may refer to Section 7 of Nguyen and Vu \cite{NV}
for instance. By
using the Laplace transform trick, Nguyen and Vu \cite
{NV} showed that
for $0<c<100$,
\begin{eqnarray*}
\mathbb{P}\biggl\{\sum_{n-s_1\leq i\leq n-1}\frac{\log\sklafrac{{\gamma
_{i+1}^2}}{n-i}}{\sqrt{2\log n}}<-
\log^{1/2+c}n\biggr\}=\mathrm{o}\bigl(\exp \bigl(-\log^{c/2}n
\bigr)\bigr),
\end{eqnarray*}
which implies (v) immediately.
\section{A replacement issue: Proof of (v)}\label{sec5}
In this section, we present the proof for (v). In other
words, we shall replace the last $s_1$ rows of Gaussian entries by
generally distributed entries. We will need the following classical
Berry--Esseen bound for sum of independent random variables. For
instance, one can refer to Theorem~5.4 of Petrov \cite{Petrov}.

\begin{lem} \label{lem.3.5}Let $Z_1,\ldots,Z_m$ be independent real
random variables such that $\mathbb{E}Z_j=0$ and $\mathbb
{E}|Z_j|^3<\infty$, $j=1,\ldots,n$. Assume that
\begin{eqnarray*}
\sigma_j^2=\mathbb{E}Z_j^2,\qquad
D_m=\sum_{i=1}^m
\sigma_j^2,\qquad  L_m=D_m^{-3/2}
\sum_{j=1}^m\mathbb{E}|Z_j|^3.
\end{eqnarray*}
Then there exists a constant $C>0$ such that
\begin{eqnarray*}
\sup_x\Biggl|\mathbb{P}\Biggl(D_m^{-1/2}\sum
_{j=1}^mZ_j\leq x\Biggr)-
\Phi(x)\Biggr|\leq CL_m.
\end{eqnarray*}
\end{lem}

Our strategy is to replace one row at each step, and derive the
difference between the distributions of the logarithms of the
magnitudes of two adjacent determinants. Hence, it suffices to compare
two matrices with only one different row. Noting that since the
magnitude of a determinant is invariant under swapping of two rows,
thus without loss of generality, we only need to compare two random
matrices $A_n=(a_{ij})_{n,n}$ and $\bar{A}_n=(\bar{a}_{ij})_{n,n}$
satisfying Assumption \textup{\ref{assC}} such that they only differ in the last row.
More precisely, we assume that $a_{ij}=\bar{a}_{ij},1\leq i\leq
n-1,1\leq j\leq n$ and $\mathbf{a}_n^T$ and $\mathbf{\bar{a}}_n^T$
are independent. Here we use $\mathbf{a}_n^T$ and $\mathbf{\bar
{a}}_n^T$ to denote the $n$th row of $A_n$ and $\bar{A}_n$,
respectively as above. Below we use the notation $\alpha_{ni}$ to
denote the cofactor of $a_{ni}$. It is elementary that
\begin{eqnarray*}
\det A_n=\sum_{k=1}^na_{nk}
\alpha_{nk},
\end{eqnarray*}
and
\begin{eqnarray*}
\det\bar{A}_n=\sum_{k=1}^n
\bar{a}_{nk}\alpha_{nk}.
\end{eqnarray*}
Now we set
\begin{eqnarray*}
\Delta=\sqrt{\alpha_{n1}^2+\cdots+
\alpha_{nn}^2}.
\end{eqnarray*}
Consider the quantities
\begin{eqnarray*}
\frac{\det A_n}{\Delta}=\sum_{k=1}^na_{nk}
\frac{\alpha
_{nk}}{\Delta},\qquad  \frac{\det\bar{A}_n}{\Delta}=\sum_{k=1}^n
\bar {a}_{nk}\frac{\alpha_{nk}}{\Delta}.
\end{eqnarray*}
By using Lemma~\ref{lem.3.5}, we obtain
\begin{eqnarray*}
\sup_{x}\biggl|\mathbb{P}\biggl\{\frac{\det A_n}{\Delta}\leq x\Bigl|\mathcal
{E}_{n-1}\biggr\}-\Phi(x)\biggr|\leq C\sum_{i=1}^n
\frac{|\alpha_{ni}|^3}{\Delta^3}.
\end{eqnarray*}
Therefore, we have
\begin{eqnarray*}
&&\sup_{x}\biggl|\mathbb{P}\biggl\{\frac{\det A_n}{\Delta}\leq x\biggr
\}-\Phi (x)\biggr|
\nonumber
\\
&&\quad = \sup_{x}\biggl|\mathbb{E}\biggl\{\mathbb{P}\biggl\{
\frac{\det A_n}{\Delta}\leq x\Bigl|\mathcal{E}_{n-1}\biggr\}\biggr\}-\Phi(x)\biggr|
\nonumber
\\
&&\quad \leq C\mathbb{E}\sum_{k=1}^n
\frac{|\alpha_{nk}|^3}{\Delta^3}.
\end{eqnarray*}
For simplicity, we will briefly denote $\mathbf{b}_k(n-1)$ and
$A_{(n-1,k)}$ by $\mathbf{b}_k$ and $A_{nk}$ respectively in the
sequel. Then by the definitions of cofactors and the Cauchy--Binet
formula, we have
\begin{eqnarray*}
\biggl(\frac{\alpha_{nk}}{\Delta} \biggr)^2&=&\frac{\det
A_{nk}^2}{\det A_{(n-1)}A_{(n-1)}^T}=\det
A_{nk}^T\bigl(A_{nk}A_{nk}^T+
\mathbf{b}_k\mathbf {b}_k^T
\bigr)^{-1}A_{nk}
\nonumber
\\
&=&\det \bigl(I_{n-1}+A_{nk}^{-1}
\mathbf{b}_k\mathbf{b}_k^T
\bigl(A_{nk}^T\bigr)^{-1} \bigr)^{-1} =
\bigl(1+\mathbf{b}_k^T\bigl(A_{nk}A_{nk}^T
\bigr)^{-1}\mathbf{b}_k \bigr)^{-1}.
\end{eqnarray*}
Moreover, one always has
\begin{eqnarray*}
\mathbb{E}\sum_{k=1}^n\frac{|\alpha_{nk}|^3}{\Delta^3}
\leq\mathbb {E}\max_{k=1,\ldots,n}\frac{|\alpha_{nk}|}{\Delta}= \mathbb {E}\max
_{k=1,\ldots,n} \bigl(1+\mathbf{b}_k^T
\bigl(A_{nk}A_{nk}^T\bigr)^{-1}
\mathbf{b}_k \bigr)^{-1/2}.
\end{eqnarray*}
Recall the definition
\begin{eqnarray*}
G_{(n-1,k)}(\alpha)=\biggl(\frac{1}n A_{nk}A_{nk}^T+
\alpha I_{n-1}\biggr)^{-1},\qquad  G_{(n-1)}(\alpha)=\biggl(
\frac{1}n A_{(n-1)}A_{(n-1)}^T+\alpha
I_{n-1}\biggr)^{-1}.
\end{eqnarray*}
By using (\ref{3.3.3.3}), we obtain
\begin{eqnarray*}
&&\mathbb{E}\max_{k=1,\ldots,n} \bigl(1+\mathbf{b}_k^T
\bigl(A_{nk}A_{nk}^T\bigr)^{-1}
\mathbf{b}_k \bigr)^{-1/2}
\nonumber
\\
&&\quad \leq\mathbb{E} \max_{k=1,\ldots,n} \biggl(1+\frac{1}n
\mathbf{b}_k^TG_{(n-1,k)}(\alpha )
\mathbf{b}_k \biggr)^{-1/2}
\nonumber
\\
&&\quad \leq \biggl(\mathbb{E}\max_{k=1,\ldots,n}\biggl(1+\frac{1}n
\mathbf {b}_k^TG_{(n-1,k)}(\alpha)
\mathbf{b}_k\biggr)^{-1} \biggr)^{1/2}
\nonumber
\\
&&\quad \leq\mathcal{O}\bigl(\log^{-4a}n\bigr).
\end{eqnarray*}
Thus we have
\begin{eqnarray*}
\sup_{x}\biggl|\mathbb{P}\biggl\{\frac{\det A_n}{\Delta}\leq x\biggr\}-
\Phi(x)\biggr|\leq C\log^{-4a}n,
\end{eqnarray*}
and\vspace*{1pt}
\begin{eqnarray*}
\sup_{x}\biggl|\mathbb{P}\biggl\{\frac{\det\bar{A}_n}{\Delta}\leq x\biggr\}-
\Phi (x)\biggr|\leq C\log^{-4a}n.
\end{eqnarray*}
Consequently, we have\vspace*{1pt}
\begin{eqnarray*}
\sup_{x}\biggl|\mathbb{P}\biggl\{\frac{\det A_n}{\Delta}\leq x\biggr\}-
\mathbb{P}\biggl\{ \frac{\det\bar{A}_n}{\Delta}\leq x\biggr\} \biggr|\leq C\log^{-4a}n,
\end{eqnarray*}
which implies\vspace*{1pt}
\begin{eqnarray*}
\sup_{x}\biggl|\mathbb{P}\biggl\{\frac{\log\det A_n^2-\log(n-1)!}{\sqrt{2\log
n}}\leq x\biggr\}-
\mathbb{P}\biggl\{\frac{\log\det\bar{A}_n^2-\log
(n-1)!}{\sqrt{2\log n}}\leq x\biggr\} \biggr|\leq C\log^{-4a}n.
\end{eqnarray*}
Then after $s_1=\lfloor\log^{3a}n\rfloor$ steps of replacing, we can
finally recover the logarithmic law to general distribution case. Thus
we completed the proof.
\begin{appendix}
\section*{Appendix A}\label{appA}
In this appendix, we provide the proof of Lemma~\ref{lem.3.1} and
Lemma~\ref{lem.3.1.1}. The proof is intrinsically the same as the
counterpart in Bai \cite{Bai}. For convenience of the
reader, we sketch it
below. For ease of notation, we represent Lemma~\ref{lem.3.1} as follows.

\renewcommand{\thelemma}{A.\arabic{lemma}}
\setcounter{lemma}{0}
\begin{lemma} Let $X=(x_{ij})_{p\times n}$ be a random matrix, where
$n-s_1\leq p\leq n$ and $\{x_{ij},1\leq i\leq p,1\leq j\leq n\}$ is a
collection of real independent random variables with means zero and
variances~1. Moreover, we assume $\sup_n\max_{i,j}\mathbb
{E}x_{ij}^4<\infty$. Let $G(\alpha)=(\frac{1}n XX^T+\alpha)^{-1}$,
we have for $\alpha= n^{-1/6}$\vspace*{1pt}
\begin{eqnarray*}
\mathbb{E}\biggl\{\frac{1}n\tr G(\alpha)\biggr\}=s_p(
\alpha)+\mathcal{O}\bigl(n^{-1/6}\bigr)
\end{eqnarray*}
and\vspace*{1pt}
\begin{eqnarray*}
\Var\biggl\{\frac{1}n\tr G(\alpha)\biggr\}=\mathcal{O}
\bigl(n^{-1/3}\bigr),
\end{eqnarray*}
where\vspace*{1pt}
\begin{eqnarray*}
s_p(\alpha)=2 \biggl(\alpha+1-\frac{i}{n}+\sqrt{\biggl[
\alpha+\biggl(1-\frac
{i}{n}\biggr)\biggr]^2+4\alpha
\frac{i}{n}} \biggr)^{-1}.
\end{eqnarray*}
\end{lemma}
\begin{pf}
For convenience, we set
\begin{eqnarray*}
r_p(\alpha)&=&\frac{1}{p}\mathbb{E}\tr G(\alpha)\\
&=&
\frac{1}{p}\mathbb {E}\sum_{k=1}^p
\frac{1}{\sklfrac{1}n x_kx_{k}^T+\alpha-\sklfrac
{1}{n^2}x_{k}X(k)^T(\sklfrac{1}n X(k)X^T(k)+\alpha I_{p-1})^{-1}X(k)x_{k}^T}.
\end{eqnarray*}
Here $x_k$ is the $k$th row of $X$ and $X(k)$ is the submatrix of $X$
by deleting its $k$th row.
Set $y_p=p/n$. Then we write
%
\begin{eqnarray}\label{5.1}
r_p(\alpha)&=&\frac{1}{p}\sum_{k=1}^p
\mathbb{E}\frac{1}{\epsilon
_k+1+\alpha-y_p+y_p\alpha r_p(\alpha)}
\nonumber
\\[-8pt]\\[-8pt]
&=&\frac{1}{1+\alpha-y_p+y_p\alpha r_p(\alpha)}+\delta, \nonumber
\end{eqnarray}
where
\begin{eqnarray*}
\epsilon_k=\frac{1}{n}\sum_{j=1}^n
\bigl(x_{kj}^2-1\bigr)+y_p-y_p
\alpha r_p(\alpha)-\frac{1}{n^2}x_kX^T(k)
\biggl(\frac{1}{n}X(k)X^T(k)+\alpha I_{p-1}
\biggr)^{-1}X(k)x_k^T,
\end{eqnarray*}
and
%
\begin{eqnarray}\label{5.2}
\delta&=&\delta_p=-\frac{1}{p}\sum
_{k=1}^p\mathbb{E}\frac
{\epsilon_k}{(1+\alpha-y_p+y_p\alpha r_p(\alpha))(1+\alpha
-y_p+y_p\alpha r_p(\alpha)+\epsilon_k)}
\nonumber
\\
&=& -\frac{1}{p}\sum_{k=1}^p
\mathbb{E}\frac{\epsilon_k}{(1+\alpha
-y_p+y_p\alpha r_p(\alpha))^2}
\\
&&{}-\frac{1}{p}\sum_{k=1}^p
\mathbb{E}\frac{\epsilon_k^2}{(1+\alpha
-y_p+y_p\alpha r_p(\alpha))^2(1+\alpha-y_p+y_p\alpha r_p(\alpha
)+\epsilon_k)}.\nonumber
\end{eqnarray}
From (\ref{5.1}), we can get
%
\begin{eqnarray}\label{5.10}
r_p(\alpha)=\frac{1}{2y_p\alpha}\bigl(\sqrt{(1+\alpha-y_p-y_p
\alpha \delta)^2+4y_p\alpha} -(1+\alpha-y_p-y_p
\alpha\delta)\bigr).
\end{eqnarray}
It is not difficult to see that
%
\begin{eqnarray}\label{5.3}
\frac{1}{|1+\alpha-y_p+y_p\alpha r_p(\alpha)+\epsilon_k|}\leq \alpha^{-1}.
\end{eqnarray}
By (\ref{5.2}) and (\ref{5.3}), we can get
%
\begin{eqnarray}\label{5.8}
|\delta|\leq\frac{1}{p}\sum_{k=1}^p
\bigl(|\mathbb{E}\epsilon_k|+\alpha ^{-1}\mathbb{E}
\epsilon_k^2\bigr) \bigl(1+\alpha-y_p+y_p
\alpha r_p(\alpha)\bigr)^2.
\end{eqnarray}
First, note that
%
\begin{eqnarray}\label{5.7}
|\mathbb{E}\epsilon_k |&=&\biggl|\mathbb{E}\biggl(y_p-y_p
\alpha\frac{1}{n}\tr G(\alpha)-\frac{1}{n^2}\tr\biggl(
\frac{1}{n}X(k)X^T(k)+\alpha I_{p-1}
\biggr)^{-1}\frac{1}{n}X(k)X^T(k)\biggr)\biggr|
\nonumber
\\
&=&\alpha\frac{1}{n}\bigl|\mathbb{E}\bigl(\tr\bigl(XX^T+\alpha
I_p\bigr)^{-1}-\tr\bigl(X(k)X^T(k)+\alpha
I_{p-1}\bigr)\bigr)\bigr|+\frac{1}{n}
\\
&=&\mathcal{O}\biggl(\frac{1}{n}\biggr).\nonumber
\end{eqnarray}
Next, we come to estimate $\mathbb{E}\epsilon_k^2$. Note that
%
\begin{eqnarray}\label{5.4}
\mathbb{E}\epsilon_k^2=\Var\{\epsilon_k
\}+(\mathbb{E}\epsilon _k)^2\leq\frac{C}{n}+T_1+T_2,
\end{eqnarray}
where
\begin{eqnarray*}
T_1&=&\mathbb{E}\biggl|\frac{1}{n^2}x_kX^T(k)
\biggl(\frac{1}{n}X(k)X^T(k)+\alpha I_{p-1}
\biggr)^{-1}X(k)x_k^T
\nonumber
\\
&&\hphantom{\mathbb{E}\biggl|}{}-\mathbb{E}^{(k)}\frac{1}{n^2}x_kX^T(k)
\biggl(\frac
{1}{n}X(k)X^T(k)+\alpha I_{p-1}
\biggr)^{-1}X(k)x_k^T\biggr|^2,\nonumber
\end{eqnarray*}
and
%
\begin{eqnarray}\label{5.11}
T_2=\frac{\alpha^2}{n^2}\mathbb{E}\biggl|\tr\biggl(\frac
{1}{n}X(k)X^T(k)+
\alpha I_{p-1}\biggr)^{-1}-\mathbb{E}\tr\biggl(
\frac{1}{n}X(k)X^T(k)+\alpha I_{p-1}
\biggr)^{-1}\biggr|^{2}.
\end{eqnarray}
Here $\mathbb{E}^{(k)}$ represents the conditional expectation given
$\{x_{ij},i\neq k\}$.
Let
\begin{eqnarray*}
\Gamma_k=\bigl(\gamma_{ij}(k)\bigr)=\frac{1}{n}X^T(k)
\biggl(\frac
{1}{n}X(k)X^T(k)+\alpha I_{p-1}
\biggr)^{-1}X(k).
\end{eqnarray*}
Then one has
%
\begin{eqnarray}\label{5.5}
T_1\leq\frac{C}{n^2}\mathbb{E}\tr\Gamma_k^2
\leq\frac{C}{n\alpha
^2}.
\end{eqnarray}
Let $\mathbb{E}_d$ be the conditional expectation given $\{
x_{ij},d+1\leq i\leq p,1\leq j\leq n\}$. Define
\begin{eqnarray*}
\gamma_d(k) &=&\mathbb{E}_{d-1}\tr\biggl(
\frac{1}{n}X(k)X^T(k)+\alpha I_{p-1}
\biggr)^{-1}-\mathbb{E}_d\tr\biggl(\frac{1}{n}X(k)X^T(k)+
\alpha I_{p-1}\biggr)^{-1}
\nonumber
\\
&=& \mathbb{E}_{d-1}\sigma_d(k)-\mathbb{E}_d
\sigma_{d}(k), \qquad d=1,2,\ldots,p,
\end{eqnarray*}
where
\begin{eqnarray*}
\sigma_d(k)=\tr\biggl(\frac{1}{n}X(k)X^T(k)+
\alpha I_{p-1}\biggr)^{-1}-\tr \biggl(\frac
{1}{n}X(k,d)X^T(k,d)+
\alpha I_{p-2}\biggr)^{-1}.
\end{eqnarray*}
Noting that
\begin{eqnarray*}
\bigl|\sigma_d(k)\bigr|\leq\alpha^{-1},
\end{eqnarray*}
one has
%
\begin{eqnarray}\label{5.6}
T_2\leq\frac{C}{n^2}\sum_{d=1}^p
\mathbb{E}\bigl|\gamma^2_d(k)\bigr|\leq \frac{C}{n\alpha^2}.
\end{eqnarray}
Combining (\ref{5.4})--(\ref{5.6}) we can get
%
\begin{eqnarray}\label{5.9}
\mathbb{E}\epsilon_k^2\leq\frac{C}{n\alpha^2}.
\end{eqnarray}
Substituting (\ref{5.7}), (\ref{5.9}) and the basic fact
\begin{eqnarray*}
\bigl(1+\alpha-y_p+y\alpha r_p(\alpha)
\bigr)^2\leq\alpha^{-2}
\end{eqnarray*}
into (\ref{5.8}) one has
\begin{eqnarray*}
|\delta|\leq\frac{C}{n\alpha^5}.
\end{eqnarray*}
Then by (\ref{5.10}) one has
\begin{eqnarray*}
r_p(\alpha) &=&\frac{1}{2y_p\alpha}\bigl(\sqrt{(1+\alpha
-y_p)^2+4y_p\alpha}-(1+\alpha-y_p)
\bigr)+\mathcal{O}\bigl(n^{-1}\alpha ^{-5}\bigr)
\nonumber
\\
&=&2\bigl(\sqrt{(1+\alpha-y_p)^2+4y_p
\alpha}+(1+\alpha -y_p)\bigr)^{-1}+\mathcal{O}
\bigl(n^{-1}\alpha^{-5}\bigr).
\end{eqnarray*}
Moreover, similar to the estimate towards (\ref{5.11}), we can get that
\begin{eqnarray*}
\Var\biggl\{\frac{1}n\tr G(\alpha)\biggr\}\leq\frac{C}{n\alpha^4}.
\end{eqnarray*}
Therefore, we can complete the proof.
\end{pf}
Now let us prove Lemma~\ref{lem.3.1.1}.\vadjust{\goodbreak}
\begin{pf*}{Proof of Lemma~\ref{lem.3.1.1}}
For the diagonal part, we have
\begin{eqnarray*}
\mathbb{E}\Biggl|\sum_{i=1}^n m_{ii}
\bigl(x_i^2-1\bigr)\Biggr|^4&=&\sum
_{j,k,u,v}m_{jj}m_{kk}m_{uu}m_{vv}
\mathbb{E}\bigl(x_{jj}^2-1\bigr) \bigl(x_{kk}^2-1
\bigr) \bigl(x_{uu}^2-1\bigr) \bigl(x_{vv}^2-1
\bigr)
\nonumber
\\
&\leq& C\biggl(\nu_8\sum_{j}m_{jj}^4+
\nu_4^2\sum_{j\neq
k}m_{jj}^2m_{kk}^2
\biggr)
\nonumber
\\
&\leq& C\bigl(\nu_8 \tr M_n^4+\bigl(
\nu_4\tr M_n^2\bigr)^2\bigr).
\end{eqnarray*}
For the off-diagonal part, we have
\begin{eqnarray*}
\mathbb{E}\biggl|\sum_{u\neq v}m_{uv}x_ux_v\biggr|^4
&=& \sum_{u_{i}\neq v_i,i=1,\ldots,4}\prod_{i=1}^4m_{u_iv_i}
\mathbb {E}\prod_{i=1}^4
x_{u_i}x_{v_i}
\nonumber
\\
&\leq&C\biggl(\nu_4^2\sum_{u\neq v}m_{uv}^4+
\nu_3^2 \nu_2\sum
_{u,v,r}m_{uv}^2m_{ur}m_{rv}+
\nu_2^4\sum_{u,v,r,w}m_{uv}m_{vr}m_{rw}m_{wu}
\biggr)
\nonumber
\\
&\leq& C \nu_4^2\bigl(\bigl(\tr M_n^2
\bigr)^2+\tr M_n^2\cdot\bigl(\tr
M_n^4\bigr)^{1/2}+\tr M_n^4
\bigr)
\nonumber
\\
&\leq& C \nu_4^2 \bigl(\tr M_n^2
\bigr)^2.
\end{eqnarray*}
Thus, we may complete the proof of Lemma~\ref{lem.3.1.1}.
\end{pf*}
\section*{Appendix B}\label{appB}
\renewcommand{\theequation}{B.\arabic{equation}}
\setcounter{equation}{0}
In this appendix, we state the proof of Proposition~\ref{pro.1.2}. We
use the idea in Nguyen and Vu \cite{NV}. First, we need the
following lemma derived
from Theorem~4.1 in G{\"o}tze and Tikhomirov \cite{GT}. Let $s_n(W_n) \leq
s_{n-1}(W_n)\leq
\cdots\leq s_1(W_n)$ be the ordered singular values of an $n\times n$
matrix $W_n$.

\renewcommand{\thelemma}{B.\arabic{lemma}}
\setcounter{lemma}{0}
\begin{lemma} \label{lem.5.5}Under the assumptions of Theorem~\ref
{thm.1.1}, there exist some positive constants $c,C,L$ such that
\begin{eqnarray*}
\mathbb{P}\bigl\{s_n(A_n)\leq n^{-L};
s_1(A_n)\leq n\bigr\}\leq\mathrm {e}^{-cn}+
\frac
{C\sqrt{\log n}}{\sqrt{n}}.
\end{eqnarray*}
\end{lemma}

\begin{remm} It is easy to get Lemma~\ref{lem.5.5} from Theorem~4.1 of
G{\"o}tze and Tikhomirov \cite{GT} by choosing $p_n=1$. We also remark here
Theorem~4.1 of
G{\"o}tze and Tikhomirov \cite{GT} are stated for more general case under
weaker moment
assumption. For convenience, we just restate it under our
setting.
\end{remm}

Now by the result of Lata{\l}a \cite{L}, it is easy to see under
the assumption
of Theorem~\ref{thm.1.1},
\begin{eqnarray*}
\mathbb{E}s_1(A_n)\leq C\biggl(\max
_i\sqrt{\sum_j
\mathbb {E}a_{ij}^2}+\max_j \sqrt
{\sum_{i}\mathbb{E}a_{ij}^2}+
\sqrt[4]{\sum_{ij}\mathbb {E}a_{ij}^4}
\biggr)\leq C\sqrt{n}.
\end{eqnarray*}
Therefore, one has
\begin{eqnarray*}
\mathbb{P}\bigl\{s_1(A_n)\geq n\bigr\}\leq
\frac{\mathbb{E}s_1(A_n)}{n}=Cn^{-1/2}.
\end{eqnarray*}
Together with Lemma~\ref{lem.5.5} we obtain
%
\begin{eqnarray} \label{5.1.1.1}
\mathbb{P}\bigl\{s_n(A_n)\geq n^{-L}\bigr\}=
1-\mathcal{O}\biggl(\frac{\sqrt{\log
n}}{\sqrt{n}}\biggr).
\end{eqnarray}
Now let $\theta_0$ follow the uniform distribution on the interval
$[-\sqrt{3},\sqrt{3}]$ independent of $A_n$. Let $\theta_{ij},1\leq
i,j\leq n$ be independent copies of $\theta_0$. And we set
$A'_n=(a'_{ij})$, where $a'_{ij}=(1-\epsilon_n^2)^{1/2}a_{ij}+\epsilon
_n \theta_{ij}$. Here we choose $\epsilon_n=n^{-(100+2L)n}$ (say).
Writing $\Theta_n=(\theta_{ij})_{n,n}$, then by Weyl's inequality,
one has
\begin{eqnarray*}
\bigl|s_i\bigl(A'_n\bigr)-\bigl(1-
\epsilon_n^2\bigr)^{1/2}s_i(A_n)\bigr|
\leq\epsilon_n\llVert \Theta _n\rrVert _{\mathrm{op}}\leq
Cn^{-(99+2L)n}.
\end{eqnarray*}
Therefore, by (\ref{5.1.1.1}) we have with probability $1-\frac{\sqrt {\log n}}{\sqrt{n}}$
\begin{eqnarray*}
\bigl|\det{A'_n}\bigr|=\prod_{i=1}^n
s_i\bigl(A'_n\bigr)=\bigl(1-
\epsilon_n^2\bigr)^{\sfrac
{n}{2}}\bigl(1+\mathcal{O}
\bigl(n^{-(99+L)n}\bigr)\bigr)^n\prod
_{i=1}^ns_i(A_n) =\bigl(1+
\mathrm{o}(1)\bigr)|\det A_n|,
\end{eqnarray*}
which implies (\ref{1.1.1.1}). Moreover, by the construction, since
$\theta_0$ is a continuous variable, it is obvious that (\ref
{1.1.1.2}) holds. Thus, we may complete the proof.

\end{appendix}

\section*{Acknowledgements}
Z.G. Bao was supported in part by NSFC Grant
11371317, NSFC Grant 11101362, ZJNSF Grant R6090034 and SRFDP Grant
20100101110001;
G.M. Pan was  supported in part by the Ministry of Education,
Singapore, under Grant \# ARC 14/11;
W. Zhou was  supported in part by the Ministry of Education,
Singapore, under Grant \# ARC 14/11, and by a Grant
R-155-000-131-112 at the National University of Singapore.



\printhistory

\end{document}